\documentclass[12pt]{amsart}

\usepackage{amsmath}
\usepackage{amssymb}
\usepackage{amsfonts}
\usepackage{amsthm}
\usepackage{enumerate}
\usepackage{hyperref}
\usepackage{color}

\theoremstyle{plain}
\newtheorem{thm}{Theorem}[section]

\newtheorem{lem}[thm]{Lemma}
\newtheorem{cor}[thm]{Corollary}
\newtheorem{conj}[thm]{Conjecture}

\theoremstyle{definition}

\newtheorem{que}[thm]{Question}

\theoremstyle{plain}
\newtheorem{rem}[thm]{Remark}

\renewcommand{\geq}{\geqslant}
\renewcommand{\leq}{\leqslant}
\renewcommand{\ge}{\geqslant}
\renewcommand{\le}{\leqslant}

\def\lref#1{Lemma~$\ref{#1}$}
\def\tref#1{Theorem~$\ref{#1}$}

\newcommand\SQS{{\rm SQS}}
\newcommand\STS{{\rm STS}}
\newcommand\TS{{\rm TS}}
\newcommand\K{\mathcal{K}}

\begin{document}

\title{Zero-sum flows for Steiner systems}

\author[S. Akbari, H.R. Maimani, L. Parsaei Majd, and I. M. Wanless]{Saieed Akbari}

\address{S. Akbari, Department of Mathematical Sciences, Sharif University of Technology, Tehran, Iran, and School of Mathematics, Institute for Research in Fundamental Sciences (IPM), P.O. Box 19395-5746.}

\email{s$_{-}$akbari@sharif.ir}

\author[]{Hamid Reza Maimani}

\address{H. R. Maimani, Mathematics Section, Department of Basic Sciences,
Shahid Rajaee Teacher Training University, P.O. Box 16785-163, Tehran,
Iran.}

\email{maimani@ipm.ir}

\author[]{Leila Parsaei Majd}

\address{L. Parsaei Majd, Mathematics Section, Department of Basic Sciences,
Shahid Rajaee Teacher Training University, P.O. Box 16785-163, Tehran,
Iran.}

\email{leila.parsaei84@yahoo.com}

\author[]{Ian M. Wanless}

\thanks{The research of the first author was partly funded by Iranian National Science Foundation (INSF) under the contract No. 96004167. 
The research of the fourth author was supported by Australian Research Council grant DP150100506}

\address{I. M. Wanless, School of Mathematics,
Monash University, Clayton Vic 3800, Australia.}

\email{ian.wanless@monash.edu}

\begin{abstract}
Given a $t$-$(v, k, \lambda)$ design, $\mathcal{D}=(X,\mathcal{B})$, a
zero-sum $n$-flow of $\mathcal{D}$ is a map $f :
\mathcal{B}\longrightarrow \{\pm1,\ldots, \pm(n-1)\}$ such that for
any point $x\in X$, the sum of $f$ over all blocks incident with $x$
is zero.  For a positive integer $k$, we find a zero-sum $k$-flow for
an $\STS(u w)$ and for an $\STS(2v+7)$ for $v\equiv 1~(\mathrm{mod}~4)$,
if there are $\STS(u)$, $\STS(w)$ and $\STS(v)$ such that the $\STS(u)$
and $\STS(v)$ both have a zero-sum $k$-flow. In 2015, it was
conjectured that for $v>7$ every $\STS(v)$ admits a zero-sum
$3$-flow. Here, it is shown that many cyclic $\STS(v)$ have a zero-sum
$3$-flow.  Also, we investigate the existence of zero-sum flows for
some Steiner quadruple systems.
\end{abstract}



\maketitle


\section{Introduction} \label{sec1}
For a graph $G$ we use $V(G)$ and $E(G)$ to denote the vertices and
edges of $G$, respectively. A \textit{zero-sum flow} of $G$ is an
assignment of non-zero real numbers to the edges of $G$ such that the
sum of the values of all edges incident with any given vertex is
zero. For a natural number $n\ge2$, a \textit{zero-sum $n$-flow} is a
zero-sum flow with values from the set $\{\pm 1,\ldots, \pm(n-1)\}$. 
For a subset $S \subseteq E(G)$, the \textit{weight} of $S$ is
defined to be the sum of the values of all edges in $S$.

A $t$-$(v, k, \lambda)$ design $\mathcal{D}$ (briefly,
\textit{$t$-design}), is a pair $(X, \mathcal{B})$, where $X$ is a
$v$-set of points and $\mathcal{B}$ is a collection of $k$-subsets of
$X$, called \textit{blocks}, with the property that every $t$-subset
of $X$ is contained in exactly $\lambda$ blocks. A $t$-$(v, k,
\lambda)$ design is also denoted by $S_{\lambda}(t, k, v)$. If
$\lambda = 1$, then $S_{\lambda}(t, k, v)$ is called a 
\textit{Steiner system}, and $\lambda$ is usually omitted.  If $t=2$
and $k=3$, then a $2$-$(v, 3, \lambda)$ design is denoted by
$\TS(v,\lambda)$, and it is called a \textit{triple system}. 
For a triple system if $\lambda=1$, then the design is called a
\textit{Steiner triple system} and is denoted by $\STS(v)$.
 
Given an indexing of the points and blocks of a $t$-design
$\mathcal{D}$ with the block set $\mathcal{B}=\{B_{1},\ldots,B_{b}\}$,
the \textit{incidence matrix} of $\mathcal{D}$ is a $v\times b ~(0, 1)$-matrix
$A=[a_{ij}]$, where
\begin{equation*}
a_{ij}= \left\{
\begin{array}{rl}
1&   \quad \text{if }x_i\in B_j, \\
0 &  \quad \text{otherwise.}
\end{array} \right.
\end{equation*}
We refer the reader to \cite{CRC} for notation and further results on designs.

Given a $t$-$(v,k,\lambda)$-design, $\mathcal{D}=(X,\mathcal{B})$, a
zero-sum $n$-flow of $\mathcal{D}$ is a map 
$f:\mathcal{B}\longrightarrow \{\pm1,\ldots, \pm(n-1)\}$ such that for
any point $x\in X$, the sum of $f$ over all blocks incident with $x$
is zero. In other words, the sum of the block weights around any point
is zero, i.e.
$$w(x)=\underset{x \in B}\sum f(B)=0.$$ This is equivalent to finding
a vector in the nullspace of the incidence matrix of the design whose
entries are all in the set $\{\pm1,\ldots, \pm(n-1)\}$. The following
theorem and two conjectures appeared in \cite{Dr.Akbari2}.

\begin{thm}\label{in3}
Every non-symmetric $2$-$(v, k, \lambda)$
design admits a zero-sum $k$-flow for some positive integer $k$.
\end{thm}

\begin{conj}\label{conj1}
Every non-symmetric design admits a zero-sum $5$-flow.
\end{conj}

\begin{conj}\label{conj2}
Every $\STS(v)$, with $v > 7$, admits a zero-sum $3$-flow.
\end{conj}

Motivated by Conjecture \ref{conj2}, in Section \ref{seccyclic} we prove that every cyclic $\STS(v)$ with $v>7$ admits a zero-sum $k$-flow for $k=3$ or $k=4$. In particular, we prove 
Conjecture \ref{conj2} for cyclic $\STS(v)$ of order $v\equiv1\pmod 6$ and $v\equiv9\pmod{18}$ and
for many cyclic $\STS(v)$ of other orders.

\medskip

For graphs $G$ and $H$, the \textit{join} of $G$ and $H$ is the graph
$G\vee H$ with vertex set $V=V(G)\cup V(H)$ and edge set $E=E(G) \cup
E(H) \cup \{uv : u\in V(G), v\in V(H)\}$.  The \textit{complete graph}
$K_n$ is the graph with $n$ vertices in which every two distinct
vertices are adjacent.
The \textit{complete bipartite graph} $K_{n, m}$ is $U\vee V$ where
$U$ and $V$ are disjoint independent sets with $\vert U\vert=n$ and
$\vert V\vert=m$.  The \textit{complete tripartite graph}
$K_{\ell,n,m}$ is $U\vee V\vee W$, where $U$, $V$ and $W$ are disjoint
independent sets with $\vert U\vert=\ell$, $\vert V\vert=n$ and 
$\vert W\vert=m$.

\section{Zero-sum flows on $\STS(vw)$ and $\STS(2v+7)$} \label{sec2}

Let $\STS(v)$ and $\STS(w)$ be two Steiner triple systems such that
the $\STS(v)$ has a zero-sum $k$-flow for $k\geq 3$. In this
section, we provide a zero-sum $k$-flow for a Steiner triple system
$\STS(v w)$. Moreover, we find a zero-sum $k$-flow for an
$\STS(2v+7)$, where $v\equiv 1~(\mathrm{mod}~4)$. 

\medskip

Our constructions will use Latin squares.
A \textit{Latin square} of order $n$ with entries from a set $X$ is an
$n\times n$ array $L$ such that every row and column of $L$ is a
permutation of $X$. Suppose that $L_1$ and $L_2$ are two Latin squares
of order $n$ with entries from $X$ and $Y$, respectively. We say that
$L_1$ and $L_2$ are \textit{orthogonal} provided that, for every $x\in X$ 
and $y \in Y$, there is a unique cell $(i, j)$ such that $L_1(i, j) = x$ 
and $L_2(i, j) = y$. Note that by \cite[p.12]{CRC} for every
positive integer $v\notin\{2, 6\}$, there are orthogonal Latin squares
of order $v$.  A \textit{transversal} of a Latin square is a set of
entries which includes exactly one representative from each row and
column and one of each symbol.

\begin{rem}\label{transversal}
It is not hard to see that a Latin square has an orthogonal mate if
and only if it can be decomposed into disjoint transversals.
\end{rem}

We refer the reader to \cite{transurv} for a survey of results on
transversals in Latin squares.

\medskip

Next, we recall the following construction for $\STS(vw)$, see
\cite{triplesystem}.

\bigskip\noindent
\textbf{Construction A.} \textbf{$\boldsymbol{\STS(vw)}$-Construction}\label{consA}
\medskip

Let $(X, \mathcal{B})$ be an $\STS(v)$ on the set $X =
\{x_1,\ldots,x_{v}\}$ and $(Y, \mathcal{B}')$ be an $\STS(w)$ on the
set $Y =\{y_1,\ldots, y_w \}$. Then define $(Z, \mathcal{C})$ as an
$\STS(vw)$ on the set $Z = \{z_{ij}, 1\leq i \leq v, 1 \leq j \leq w\}$ 
with two types of blocks as follows:

For $j=1,\ldots, w$, consider a copy $K^{j}_{v}$ of the complete
graph $K_{v}$, with vertex set $\{z_{1j},\ldots, z_{v j}\}$. Using
$\mathcal{B}$, one can partition the edges of each $K^{j}_{v}$ into
triangles, for $j=1,\ldots, w$. We say that the blocks made by these
triangles are of Type A. Now, consider the complete graph $K_w$ with
vertex set $K^{j}_{v}$ for $1 \leq j \leq w$. Using $\mathcal{B}'$ one
can partition the edges of $K_w$ into triangles.  Join every vertex of
$K^{i}_{v}$ to every vertex of $K^{j}_{v}$, for $1 \leq i<j \leq w$. 
Using the partition of $K_w$, every triangle in $K_w$ corresponds
to a complete tripartite graph $K_{v,v,v}$ which has $3v^2$
edges. Now, for each triangle $\{K^{p}_{v}, K^s_{v}, K^t_{v}\}$ of $K_w$, 
where $1 \leq p<s<t\leq w$, consider a Latin square $L=L(p,s,t)$ 
of order $v$ on the set $\{z_{1t},\ldots, z_{v t}\}$ such that the
rows and columns are indexed by $\{z_{1p},\ldots, z_{v p}\}$ and
$\{z_{1s},\ldots,z_{vs}\}$, respectively. For $1 \leq i\leq v$ 
and $1 \leq j\leq v$, we make a block $\{z_{ip}, z_{js}, L(z_{ip},z_{js}) \}$ 
of Type B. It is not hard to see that all blocks
of Type A and Type B together form an $\STS(v w)$.

\medskip

This construction allows us to prove the following lemma.

\begin{lem}
  Let $v$ and $w$ be two positive integers for which there exist
  $\STS(v)$ and $\STS(w)$, where at least one of the $\STS(v)$ and
  $\STS(w)$ has a zero-sum $k$-flow for some $k\geq 3$. Then there
  exists an $\STS(v w)$ which has a zero-sum $k$-flow.
\end{lem}

\begin{proof}
Suppose that an $\STS(v)$ has a zero-sum $k$-flow for $k\geq 3$. In
Construction A, we let the blocks of Type A inherit a zero-sum
$k$-flow from the $\STS(v)$. According to Remark \ref{transversal},
since $v\notin\{2,6\}$, in Construction A one can choose Latin
squares that decompose into transversals $T_1,\ldots,T_v$, each of
which corresponds to a collection of blocks in the $\STS(vw)$.  Now,
assign values $+2, -1, -1$ to the blocks from $T_1, T_2, T_3$,
respectively. Then, label the blocks from $T_i$ with $(-1)^i$ for
$i=4,\ldots,v$. In this way, the Type B blocks defined by each Latin
square contribute a total of zero to the weight of every vertex.
\end{proof}

\medskip

We need the following observation to prove our next results. This can be
found in \cite[p.41]{ed2}.

\begin{rem}\label{decom}
For odd $v$, the edges of $K_{v+7}$  can be partitioned into $v+7$
triangles and $v\,$ $1$-factors. Note that each vertex appears in
exactly three triangles.
\end{rem}

\bigskip\noindent
\textbf{Construction B. $\boldsymbol{\STS(2v+7)}$-Construction}\label{consB}
\medskip

Let $(X, \mathcal{A})$ be a Steiner triple system of order $v$, with
$X = \{x_1,\ldots,x_{v}\}$, and let $Y$ be a set of size $v + 7$, such
that $X \cap Y =\varnothing$. Using Remark \ref{decom}, partition the
edges of $K_{v+7}$ with vertex set $Y$ into a set $L$ containing $v+7$
triangles and a set $F =\{F_1,\ldots, F_v\}$ containing $v$
$1$-factors. Set $Z = X \cup Y$ and define a collection of triples
$\mathcal{B}$ as follows: We can consider a block corresponding to
each triangle in $L$. Put all such blocks in a set $N$. Now, join
$x_i$ to the end vertices of each edge of $F_i$, for $i=1,\ldots, v$,
to obtain some new triangles. Let $T$ be a set of blocks corresponding
to these new triangles. Then, $(Z,\mathcal{B})$ is a Steiner triple
system of order $2v + 7$, where $\mathcal{B}=\mathcal{A}\cup N \cup T$.  
See \cite[p.41--42]{ed2}.

\begin{rem}\label{decompose}
Let $n\geq 8$ be an even positive integer, and let
$Y=\{y_1,\ldots,y_{n}\}$. It is clear that $n=v+7$, for some odd
$v\geq 1$. We know that the edges of $K_n$, with vertex set $Y,$ can be
partitioned into $n$ triangles and $v$ $1$-factors, $\{F_1,\ldots,
F_v\}$. If we assign the value $1$ to each of the $n$ triangles, then
the sum of the values of the three triangles containing $y_i$ is $3$,
for $i=1,\ldots, n$.
  
Now, if $v=1$, then we have just one $1$-factor, $F_1$. Assign $-3$ to
each edge of $F_1$. Otherwise, $v\ge3$. Assign $-1$ to the edges of
$F_1$, $F_2$ and $F_3$. Then assign $(-1)^j$ to $F_j$ for
$j=4,\ldots,v$.  Since $v$ is odd, in all cases the sum of the values
of the edges in $\cup_{j=1}^v F_j$ incident with $y_i$ is $-3$, for
$i=1,\ldots, n$. Hence the total weight allocated to the edges and 
triangles incident with any vertex in $Y$ is $0$.
\end{rem}
 
Next, from a zero-sum $k$-flow for $\STS(v)$, we show how to obtain a
zero-sum $k$-flow for an $\STS(2v+7)$, if $v\equiv
1~(\mathrm{mod}~4)$. We say that a graph $G$ has a \textit{k-null
  1-factorisation} if $G$ has a zero-sum $k$-flow and there is a
$1$-factorisation in which the weight of each $1$-factor is zero. We
call each $1$-factor in a $k$-null $1$-factorisation of $G$ a
\textit{$k$-null $1$-factor}. We use the following lemma, see the
proof of Lemma 4.2 in \cite{Dr.Akbari}.

\begin{lem}\label{kn,n}
There exists a $3$-null $1$-factorization of $K_{n,n}$ for every
$n\geq 3$. If $n$ is even and $n\neq 6$, then $K_{n,n}$ has a $2$-null
$1$-factorization.
\end{lem}

\begin{thm}\label{2v+7}
Let $v>9$ be a positive integer and $v\equiv 1~(\mathrm{mod}~4)$. If
there exists an $\STS(v)$ with a zero-sum $k$-flow for some positive
integer $k\geq 2$, then there exists an $\STS(2v +7)$ with a zero-sum
$k$-flow.
\end{thm}

\begin{proof}
Let $(X, \mathcal{A})$ be an $\STS(v)$, with $X=\{x_1,\ldots,x_{v}\}$,
which has a zero-sum $k$-flow, and let $Y$ be a set of size $v + 7$
such that $X\cap Y=\varnothing$. Keep the values of the blocks in
$\mathcal{A}$.  Consider the Steiner triple system on $X \cup Y$ given
in Construction B. Since $v\equiv 1~(\mathrm{mod}~4)$ and $v>9$, we
know that $v+7=4s$ for some integer $s\ge5$.  Let $2s=t+7$, for some
odd $t\ge3$. We have $K_{v+7}=\K\vee\K'$, where $\K$ and $\K'$ are both
copies of $K_{t+7}$.  By Remark \ref{decom} we can decompose the edges
of $\K$ into $1$-factors $M_{1},\ldots,M_{t}$ and $t+7$ triangles.  We
give each of these triangles a weight of $1$.  For $1\le i\le t$ and
for each edge $e$ in $M_i$ we then make a new block containing $x_i$
and the end vertices of $e$. We assign this block a weight equal to
the value that $e$ was assigned in Remark \ref{decompose}.  We then
decompose $\K'$ in a similar way into $t+7$ triangles and 1-factors
$M'_{1},\ldots,M'_{t}$.  We allocate a weight of $-1$ to the $t+7$
triangles and we give each edge in $M'_i$ the negative of the weight
that the edges in $M_i$ were given.  In this way, when we join $x_i$
to $M'_i$ in the same way that we joined $x_i$ to $M_i$, the total weight
of the blocks incident with $x_i$ will be zero for $1\le i\le
t$. Similarly, Remark \ref{decompose} shows that for any vertex in
$Y,$ there is zero total weight for the blocks so far constructed that
are incident with that vertex.

The edges between $\K$ and $\K'$ form
a $K_{t+7,t+7}$, which has a $2$-null $1$-factorization
$F_1,\ldots, F_{t+7}$, by \lref{kn,n}.  For $i=1,\ldots,v-t$ and for
each edge $e'$ in $F_i$, make a new block containing $x_{t+i}$ and the
end vertices of $e'$. Assign this block a weight equal to the value
that $e'$ received in the $2$-null $1$-factorization. By this process
we obtain a zero-sum $k$-flow for the $\STS(2v+7)$ formed
by Construction~B.
\end{proof}

\begin{rem}
If $v=9$ and there exists an $\STS(9)$ with a zero-sum $3$-flow, then
we are not able to find a zero-sum $3$-flow for the $\STS(25)$ obtained
by Construction B. This is because, in Remark $\ref{decompose}$ we utilised
a weight of $-3$ in the case when $t=1$.
Note that in this case, we can find a zero-sum $4$-flow for the constructed
$\STS(25)$. However, in \cite{Dr.Akbari} it was proved that for every
pair $(v, \lambda)$ such that a $\TS(v, \lambda)$ exists, there is one
with a zero-sum $3$-flow, except when 
$(v, \lambda) \in \{(3,1), (4,2), (6, 2), (7, 1)\}$.
\end{rem}

It would be interesting to know if the restriction to $v\equiv1~(\mathrm{mod}~4)$ is 
really needed in \tref{2v+7}.

\begin{que}
Let $v,k$ be positive integers such that $v\equiv3~(\mathrm{mod}~4)$
and $k\ge2$.  Suppose that in Construction B we use an $\STS(v)$ that
has a zero-sum $k$-flow.  Is there necessarily a zero-sum $k$-flow for
the resulting $\STS(2v+7)$?
\end{que}

\section{Flows in cyclic STS}\label{seccyclic}

In this section we are going to verify that for $v>7$ each cyclic
$\STS(v)$ has a zero-sum $4$-flow and that many such systems have
a zero-sum $3$-flow. First we need some definitions.

An \textit{automorphism} of a $t$-$(v,k,\lambda)$ design, $(X,
\mathcal{B})$, is a bijection $\alpha : X \longrightarrow X$ such that
$B =\{x_1,\ldots, x_k\} \in \mathcal{B}$ if and only if 
$B\alpha =\{x_1\alpha, x_2\alpha,\ldots, x_k\alpha\} \in \mathcal{B}$. A
$t$-$(v,k,\lambda)$ design is called \textit{cyclic} if it has an
automorphism that is a permutation consisting of a single cycle of
length $v$; this automorphism is called a \textit{cyclic
  automorphism}.  Throughout, we will assume for our cyclic
$t$-$(v,k,\lambda)$ design that $X = \mathbb{Z}_v$, and $\alpha : i
\longrightarrow i+1 ~(\mathrm{mod}~v)$ is its cyclic automorphism. The
blocks of a cyclic $t$-$(v,k,\lambda)$ design are partitioned into
orbits under the action of the cyclic group generated by
$\alpha$. Each orbit of blocks is completely determined by any of its
blocks, and $\mathcal{B}$ is determined by a collection of blocks
called \textit{base blocks} (sometimes also called \textit{starter blocks} 
or \textit{initial blocks}) containing one block from each
orbit.  For an example, $X = \{1, 2, 3, 4, 5, 6, 7\}$ and
$$\mathcal{B} = \big\{ \{1, 2, 4\}, \{2, 3, 5\}, \{3, 4, 6\}, \{4, 5, 7\}, \{5, 6, 1\}, \{6, 7, 2\}, \{7, 1, 3\} \big\},$$
form an $\STS(7)$ which is cyclic, since the
permutation $\alpha = (1234567)$ is an automorphism. 

\bigskip

In 1939, Rose Peltesohn solved both of Heffter's Difference Problems,
see \cite{cyclicsts}. This solution provides the following theorem,
see \cite[Section 1.7]{ed2}.

\begin{thm}
For all $v\equiv 1~\text{or}~3~(\mathrm{mod}~6)$ with $v \neq 9$, there
exists a cyclic $\STS(v)$.
\end{thm}

\begin{rem}
If $v\equiv 1~(\mathrm{mod}~6)$, every cyclic $\STS(v)$ has
$\frac{v-1}{6}$ full orbits. Also, if $v\equiv 3~(\mathrm{mod}~6)$,
every cyclic $\STS(v)$ has $\frac{v-3}{6}$ full orbits and one short
orbit which contains the block $\{0, \frac{v}{3}, \frac{2v}{3}\}$.
Moreover, note that every full orbit contains each point $3$ times,
and each point appears once in the short orbit, see \cite{triplesystem}.
\end{rem}

For $v\equiv 3~(\mathrm{mod}~6)$, we will classify orbits of a cyclic $\STS(v)$
into three types. For $i=1,2,3$ an orbit is of Type $i$ if every block
in the orbit contains representatives of precisely $i$ different
congruence classes modulo $3$. As $v$ is divisible by $3$, every orbit
will be of Type $1$, Type $2$ or Type $3$ and its type can be established
by examining any single block in the orbit.

Since the incidence matrix of $\STS(7)$ has full rank, $\STS(7)$ has no zero-sum $k$-flow. Also, by \cite[Section 1.7]{ed2}, there is no
cyclic $\STS(9)$.  In the following we are going to show that every
cyclic $\STS(v)$ for $v>7$ admits a zero-sum $k$-flow for $k=3$ or $k=4$.

We will split the $v\equiv 3~(\mathrm{mod}~6)$ case into three subcases:
$v\equiv 3, 9 ~\text{or}~15~(\mathrm{mod}~18)$. In the following we
prove that if $v\equiv 1~(\mathrm{mod}~6)$ or $v\equiv 9
~(\mathrm{mod}~18)$ and $v\neq 7$, then each cyclic $\STS(v)$ admits a
zero-sum $3$-flow. In other words, Conjecture \ref{conj2} is true for
these families of Steiner triple systems. Also, we show that for
$v\equiv 3 ~\text{or}~ 15~(\mathrm{mod}~18)$, each cyclic $\STS(v)$
has a zero-sum $4$-flow. We need the following lemmas to prove
our main results.

\begin{lem}\label{v9}
For $v\equiv 9~(\mathrm{mod}~18)$, every cyclic $\STS(v)$ has a full
orbit of Type $3$.
\end{lem}

\begin{proof}
Suppose that there exists a cyclic $\STS(v)$, $S$, with no full orbit of Type $3$. Let $S$ have $t$ full orbits of Type $2$ and $s$ full orbits of Type $1$. Note that $t$ and $s$ are two non-negative integers and $t+s=({v-3})/{6}$.
Now, count the number of pairs $\{a, b\}$ where $a\not\equiv b~(\mathrm{mod}~3)$, among all blocks of $S$. Since the short orbit has Type $1$, and every full orbit has $v$ blocks, we obtain the following equality:
$$2 v t =3\frac{v}{3} \times\frac{v}{3}.$$
Hence $t={v}/{6}$, a contradiction.
\end{proof}

\begin{lem}\label{v3,15}
  Let $v\equiv 3 ~\text{or}~ 15~(\mathrm{mod}~18)$ and $S$ be a cyclic
  $\STS(v)$ with no full orbit of Type $3$. Then $S$ has no full orbit
  of Type $1$.
\end{lem}

\begin{proof}
Suppose $S$ has $t$ full orbits of Type $2$ and $s$ full orbits of
Type $1$. We have $t+s=({v-3})/{6}$.  Since $v/3$ is not divisible by
$3$, the short orbit has Type $3$.  Now, count the number of pairs
$\{a, b\}$ in all blocks of $S$, where $a\not\equiv b~(\mathrm{mod}~3)$.
We have
$$2tv+3\frac{v}{3}=3\frac{v}{3}\times\frac{v}{3}.$$
Hence, $t=({v-3})/{6}$ and $s=0$.
\end{proof}

\begin{rem}\label{mod3} 
Let $v\equiv 9~(\mathrm{mod}~18)$, and suppose that a cyclic $\STS(v)$
has a full orbit of Type $3$ generated from a base block
$\{a,b,c\}$. Then the blocks $\{a+3i, b+3i, c+3i\}$ for $0\leq i\leq
\frac{v}{3}-1$, contain exactly one occurrence of each point in
$\mathbb{Z}_v$. This is because $\{a+3i:0\leq i\leq \frac{v}{3}-1\}$
contains the $v/3$ points that are congruent to $a\pmod 3$. Similar
statements holds for $\{b+3i\}$ and $\{c+3i\}$, and these sets are
disjoint because the orbit is of Type $3$.
\end{rem}

Using Lemmas \ref{v9} and \ref{v3,15}, and Remark \ref{mod3}, we have
the following theorems about the existence of a zero-sum $k$-flow with
$k=3$ or $k=4$, for every cyclic $\STS(v)$.

\begin{thm}\label{zs3}
Every cyclic $\STS(v)$ for $v\equiv 1~(\mathrm{mod}~6)$ or $v\equiv 9 ~(\mathrm{mod}~18)$ with $v\neq 7$ admits a zero-sum $3$-flow.
\end{thm}

\begin{proof}
There is no cyclic $\STS(9)$, so $v> 9$ and we have at least two full orbits. 
The case when $v\equiv 1~(\mathrm{mod}~6)$ is handled by 
\cite[Theorem 1.7]{Dr.Akbari}, so we assume that 
$v\equiv 9 ~(\mathrm{mod}~18)$. In this case,
by \lref{v9}, there exists a full orbit with a block $\{a, b, c\}$
congruent to $\{0, 1, 2\}~(\mathrm{mod}~3)$. So, assign the weight of
all blocks within a full orbit of Type $3$ as follows:
$$-1, +1, +1, -1, +1, +1, -1, +1, +1,\ldots.$$ Note that by Remark
\ref{mod3}, each point gets weight $+1$ along this orbit. Now, if
$O_2, O_3,\ldots, O_{\frac{v-3}{6}}$ are the other full orbits, assign
weight $(-1)^{i+1}$ to every block $O_i$, for $2\leq i \leq
\frac{v-3}{6}$.  If $~\frac{v-3}{6}~$ is odd, assign weight $-1$ to
the blocks in the short orbit. Otherwise, assign value $2$ to the
blocks in the short orbit.
\end{proof}

\medskip

For the cases not covered by \tref{zs3}, we have the following result.

\begin{thm}\label{remcases}
  Suppose that $S$ is a cyclic $\STS(v)$, where $v\equiv
  3~\text{or}~15~(\mathrm{mod}~18)$ and $v>3$.  Then $S$ has a
  zero-sum $4$-flow.  If $S$ has any full orbit of Type $1$ or Type
  $3$, then $S$ has a zero-sum $3$-flow.
\end{thm}

\begin{proof}
We first show that $S$ admits a zero-sum $4$-flow.  Assign value $-3$
to the blocks in the short orbit. For the first full orbit, assign a
value of $2$ if there are an even number of full orbits, and a value
of $1$ otherwise. For the other full orbits, alternate between
assigning $-1$ and $1$ to the orbit. This produces a zero-sum $4$-flow
for $S$.  If $S$ has a full orbit of Type $3$, then similar to the proof
of \tref{zs3}, there exists a zero-sum $3$-flow for $S$. By
\lref{v3,15}, we know that if some full orbit has Type $1$ then there
will be a full orbit of Type $3$, so we are also done in that case.
\end{proof}

\begin{cor}
Every cyclic $\STS(v)$ with $v>7$ admits a zero-sum $4$-flow.
\end{cor}

We stress that Theorem \ref{remcases} does not rule out the existence
of a zero-sum $3$-flow for a cyclic $\STS(v)$ that has no full orbits
of Type $1$ or $3$.  Such triple systems do exist.  For example, any
triple system built using three identical cyclic quasigroups in the
Bose Construction (\cite[Section 1.2]{ed2}), will have only full
orbits of Type $2$. We next show that such STS may still have a zero-sum
$3$-flow.  There are two cyclic $\STS(15)$. The cyclic $\STS(15)$ with
the base blocks $\{0, 1, 4\}$, $\{0, 2, 8\}$ and $\{0, 5, 10\}$ is not
obtained from the Bose construction, but the other one constructed by
the base blocks $\{0, 1, 4\}$, $\{0, 2, 9\}$ and $\{0, 5, 10\}$ arises from
the Bose construction. However, both of them admit a zero-sum $3$-flow
and the full orbits of these cyclic $\STS(15)$ are all of Type $2$.

In the following one can find a zero-sum $3$-flow for the cyclic
$\STS(15)$ with the base blocks $\{0, 1, 4\}$, $\{0, 2, 8\}$ and 
$\{0,5, 10\}$.  The fourth number (after each block) is the flow value
assigned to that block. We omit the $\{$~$\}$ symbols in each block.
\begin{equation*}
\begin{array}{c@{\hskip2mm}c@{\hskip2mm}c@{\hskip5mm}r@{\hskip1cm}c@{\hskip2mm}c@{\hskip2mm}c@{\hskip5mm}r@{\hskip1cm}c@{\hskip2mm}c@{\hskip2mm}c@{\hskip5mm}r}
0&1&4&\textcolor{red}{-1}&0&2&8&\textcolor{red}{1}&0&5&10&\textcolor{red}{2}\\
1&2&5&\textcolor{red}{-1}&1&3&9&\textcolor{red}{1}&1&6&11&\textcolor{red}{2}\\
2&3&6&\textcolor{red}{1}&2&4&10&\textcolor{red}{-1}&2&7&12&\textcolor{red}{2}\\
3&4&7&\textcolor{red}{-1}&3&5&11&\textcolor{red}{-1}&3&8&13&\textcolor{red}{2}\\
4&5&8&\textcolor{red}{1}&4&6&12&\textcolor{red}{-1}&4&9&14&\textcolor{red}{2}\\
5&6&9&\textcolor{red}{-1}&5&7&13&\textcolor{red}{-1}\\
6&7&10&\textcolor{red}{1}&6&8&14&\textcolor{red}{-1}\\
7&8&11&\textcolor{red}{-1}&7&9&0&\textcolor{red}{1}\\
8&9&12&\textcolor{red}{-1}&8&10&1&\textcolor{red}{-1}\\
9&10&13&\textcolor{red}{-1}&9&11&2&\textcolor{red}{-1}\\
10&11&14&\textcolor{red}{1}&10&12&3&\textcolor{red}{-1}\\
11&12&0&\textcolor{red}{-1}&11&13&4&\textcolor{red}{1}\\
12&13&1&\textcolor{red}{1}&12&14&5&\textcolor{red}{1}\\
13&14&2&\textcolor{red}{-1}&13&0&6&\textcolor{red}{-1}\\
14&0&3&\textcolor{red}{-1}&14&1&7&\textcolor{red}{-1}\\
\end{array}
\end{equation*}
Also, a cyclic $\STS(15)$ with the base
blocks $\{0, 1, 4\}$, $\{0, 2, 9\}$ and $\{0, 5, 10\}$ has a zero-sum $3$-flow as follows:
\begin{equation*}
\begin{array}{c@{\hskip2mm}c@{\hskip2mm}c@{\hskip5mm}r@{\hskip1cm}c@{\hskip2mm}c@{\hskip2mm}c@{\hskip5mm}r@{\hskip1cm}c@{\hskip2mm}c@{\hskip2mm}c@{\hskip5mm}r}
0&1&4&\textcolor{red}{1}&0&2&9&\textcolor{red}{-1}&0&5&10&\textcolor{red}{1}\\
1&2&5&\textcolor{red}{-2}&1&3&10&\textcolor{red}{-2}&1&6&11&\textcolor{red}{1}\\
2&3&6&\textcolor{red}{1}&2&4&11&\textcolor{red}{2}&2&7&12&\textcolor{red}{1}\\
3&4&7&\textcolor{red}{-2}&3&5&12&\textcolor{red}{1}&3&8&13&\textcolor{red}{1}\\
4&5&8&\textcolor{red}{-1}&4&6&13&\textcolor{red}{2}&4&9&14&\textcolor{red}{-1}\\
5&6&9&\textcolor{red}{-2}&5&7&14&\textcolor{red}{2}\\
6&7&10&\textcolor{red}{1}&6&8&0&\textcolor{red}{-2}\\
7&8&11&\textcolor{red}{-2}&7&9&1&\textcolor{red}{2}\\
8&9&12&\textcolor{red}{1}&8&10&2&\textcolor{red}{1}\\
9&10&13&\textcolor{red}{2}&9&11&3&\textcolor{red}{-1}\\
10&11&14&\textcolor{red}{-2}&10&12&4&\textcolor{red}{-1}\\
11&12&0&\textcolor{red}{1}&11&13&5&\textcolor{red}{1}\\
12&13&1&\textcolor{red}{-2}&12&14&6&\textcolor{red}{-1}\\
13&14&2&\textcolor{red}{-2}&13&0&7&\textcolor{red}{-2}\\
14&0&3&\textcolor{red}{2}&14&1&8&\textcolor{red}{2}\\
\end{array}
\end{equation*}

\section{Steiner Quadruple Systems}

In this section we study zero-sum $k$-flows in Steiner quadruple
systems (SQS). For $k\ge3$ we show the following results.  If we have a
zero-sum $k$-flow for two $\SQS(v)$, then we can find a zero-sum
$k$-flow for an $\SQS(2v)$. Also, if there are an $\SQS(u)$ and an
$\SQS(v)$ both with a zero-sum $k$-flow, then we can find a zero-sum
$k$-flow for an $\SQS(uv)$.

First we recall some definitions and background about Steiner
quadruple systems from \cite{surveysqs} and \cite{survey2}. A
\textit{Steiner quadruple system} (or simply a quadruple system) is a
pair $(X, \mathcal{B})$ which is a $3$-design with parameters $(v, 4,
1)$ such that any $3$-subset of $X$ belongs to exactly one block of
$\mathcal{B}$. A Steiner quadruple system of order $v$ is denoted by
$\SQS(v)$.  One obtains immediately that $v\equiv 2$ or $4
~(\mathrm{mod}~6)$ is a necessary condition for the existence of an
$\SQS(v)$. The total number of quadruples is
$\frac{1}{24}v(v-1)(v-2)$, the number of quadruples containing a given
element is $\frac{1}{6}(v-1)(v-2)$, and the number of quadruples
containing a given pair of elements is $\frac{1}{2}(v-2)$.  In 1960,
Hanani \cite{Hanani2} proved that the set of possible orders for
quadruple systems consists of all positive integers $v \equiv 2$ or
$4~(\mathrm{mod}~6)$. If $(X, \mathcal{B})$ is a quadruple system and
$x$ is any element in $X$, put $X_{x}=X\setminus \{x\}$ and
$\mathcal{B}(x)=\{B\setminus \{x\}: B\in \mathcal{B}, x\in B\}$. It
can be easily checked that $(X_{x}, \mathcal{B}(x))$ is a Steiner
triple system which is called a \textit{derived triple system} of the
quadruple system $(X, \mathcal{B})$.

We now recall two recursive constructions of $\SQS(2v)$
and $\SQS(uv)$ from \cite{surveysqs}.

\bigskip\noindent
\textbf{Construction C. $\boldsymbol{\SQS(2v)}$-Construction}
\medskip

Let $v \equiv 2$ or $4~(\mathrm{mod}~6)$. Consider two disjoint copies
of $K_v$, with vertex sets $X$ and $Y$ such that $\vert X\vert=\vert Y\vert=v$.
Let $(X, \mathcal{A})$ and $(Y, \mathcal{B})$ be any two
$\SQS(v)$. Let $F= \{F_1,\ldots, F_{v-1}\}$ and $G = \{G_1,\ldots,
G_{v-1}\}$, be two $1$-factorizations of $K_{v}$ on $X$ and $Y,$
respectively.
Assume that $\mathcal{C}=\mathcal{A}\cup\mathcal{B}\cup T$ on the
point set $Z = X\cup Y$, where the elements of $T$ are defined as
follows:

If $x_1, x_2 \in X$ and $y_1, y_2 \in Y$, then $\{x_1, x_2, y_1, y_2\}\in T$ 
if and only if there exists $i$, with $1\leq i \leq v-1$ such
that $x_1 x_2$ and $y_1 y_2$ are edges in $F_i$ and $G_i$, respectively.
It is shown in \cite{surveysqs} that $(Z, \mathcal{C})$ is an $\SQS(2v)$. 

\bigskip

In the following lemma, we assume that there are two $\SQS(v)$ with a
zero-sum $k$-flow. Then, we find a zero-sum $k$-flow for an $\SQS(2v)$.

\begin{lem}\label{sqs(2v)}
Let $(X, \mathcal{A})$ and $(Y, \mathcal{B})$ be two $\SQS(v)$ with
$X\cap Y=\varnothing$, where both $\SQS(v)$ have a zero-sum $k$-flow
for $k\geq 3$. Then there is an $\SQS(2v)$ with a zero-sum $k$-flow.
\end{lem}

\begin{proof}
In Construction C, we keep the values of all blocks in
$\mathcal{A}\cup\mathcal{B}$. Hence, it only remains to define weights
for the blocks in $T$.
First, we assign $2$, $-1$ and $-1$, to the elements of $F_1$,
$F_2$, and $F_3$, respectively, and assign $(-1)^i$ to $F_i$, for
$4\leq i\leq v-1$.  Note that $v-1$ is odd.  
Now, each block of $T$ contains exactly one element
of one of the $F_i$, so we may assign the value of that element to the
block. In this way, we obtain a zero-sum $3$-flow for an $\SQS(2v)$.
\end{proof}

\bigskip\noindent
\textbf{Construction D. $\boldsymbol{\SQS(uv)}$-Construction}
\medskip

Let $(X, \mathcal{A})$ and $(Y, \mathcal{B})$ be an $\SQS(u)$ and an
$\SQS(v)$, respectively, and consider the following properties: Define
a ternary operation $\langle ~, ~, ~ \rangle$ on $X$ by 
$\langle a,b, c\rangle=d$ whenever $\{a, b, c, d\}\in \mathcal{A}$, and 
$\langle a, a, b\rangle=b$.  Now, denote $X_y=X\times \{y\}$, and for
every $y\in Y$, let $\mathcal{A}_y$ be a collection of quadruples on
$X_y$ such that $(X_y, \mathcal{A}_y)$ is an $\SQS(u)$. Let
$Y=\{y_1,\ldots, y_{v}\}$, and $F^{(y_i)}=\{F_{1}^{(y_i)},
F_{2}^{(y_i)},\ldots, F_{u-1}^{(y_i)}\}$ for $i\in \{1,\ldots, v\}$,
be a $1$-factorization of $K_{u}$ on $X_{y_i}$. For the set $X\times Y$ 
define the following collection $\mathcal{C}$ of quadruples:

\begin{itemize}

\item[(1)] $\mathcal{C}$ contains every quadruple belonging to
  $\mathcal{A}_{y_i}$ for any $y_i \in Y.$

\item[(2)] If $(a, y_i), (b, y_i)\in X_{y_i}$ and 
  $(c, y_j), (d,y_j)\in X_{y_j}$ for $i<j$, then
$$\{(a, y_i), (b, y_i), (c, y_j), (d, y_j)\}\in \mathcal{C}$$ 
if and only if $(a, y_i)(b, y_i)$ and $(c, y_j)(d, y_j)$ are edges
in $F_{k}^{(y_i)}$ and $F_{k}^{(y_j)}$, respectively, for some $1\leq k\leq u-1$.

\item[(3)] For every quadruple $\{y_i, y_j, y_t, y_s\}\in \mathcal{B}$
  and for every three (not necessarily distinct) elements $a, b, c\in X$, 
  $\mathcal{C}$ contains 
  $\{(a, y_i), (b, y_j), (c, y_t), (\langle a, b, c\rangle, y_s)\}$ 
  where $i<j<t<s$.

\end{itemize}
It is shown in \cite{surveysqs} that $(X\times Y, \mathcal{C})$ is an
$\SQS(uv)$.

\medskip

In the following lemma we present a zero-sum $k$-flow for an $\SQS(uv)$
using Construction D.

\begin{lem}
  Let $(X, \mathcal{A})$ and $(Y, \mathcal{B})$ be an $\SQS(u)$ and an
  $\SQS(v)$, respectively, both having
  a zero-sum $k$-flow for some $k\geq 3$. Then there is
  an $\SQS(uv)$ which admits a zero-sum $k$-flow.
\end{lem}

\begin{proof}
In Construction D, one can ignore the blocks from $(1)$ because they
inherit their value from the zero-sum flow of the $\SQS(u)$.  It is
not hard to see that there exists a zero-sum $3$-flow on the blocks
from $(2)$, by treating them as a complete bipartite graph similar to
the proof of \lref{sqs(2v)}.  That leaves the blocks from $(3)$, where
for each given block of $\mathcal{B}$ we have $u^3$ quadruples in
$\SQS(uv)$ because we have $u$ choices for each of $a, b$ and
$c$. There are exactly $u^2$ blocks obtained from a given block
$\{y_i, y_j, y_t, y_s\}\in \mathcal{B}$ that contain an element
$(a,y_i)$ for any fixed $a\in X$. Now, assign to all blocks obtained
from $\{y_i, y_j, y_t, y_s\}$, the weight of the block $\{y_i, y_j,
y_t, y_s\}$ in the zero-sum $k$-flow for the $\SQS(v)$. In this way we
obtain an $\SQS(uv)$ with a zero-sum $k$-flow.
\end{proof}

A $t$-design $(X, \mathcal{B})$ is said to be
$\alpha$-\textit{resolvable} if there exists a partition of the
collection $\mathcal{B}$ into parts called
$\alpha$-\textit{parallel classes}
(or $\alpha$-\textit{resolution classes}) such that each point of $X$
occurs in exactly $\alpha$ blocks in each class. When $\alpha=1$,
$\alpha$ is omitted.  We denote the number of $\alpha$-parallel
classes by $\rho=r/\alpha$, where $r$ is the number of appearances of
each point $x\in X$ among the blocks of the design.  A
$t$-$(v,k,\lambda)$ design is called an \textit{even design} when it
is $\alpha$-resolvable with even $\rho$.
Moreover, a $t$-$(v,k,1)$ design, $S(t, k, v)$, is called
$i$-\textit{partitionable} (some literature uses the alternative term
$i$-\textit{resolvable}, but to avoid confusion we will not) if the block
set can be partitioned into $S(i, k, v)$ designs for $0<i<t$.
Note that by \cite[Section 11]{surveysqs}, if $\alpha=i=2$,
then $2$-resolvability and $2$-partitionability are the same for
$SQS(v)$.  We refer the reader to \cite{someconstruction} for more
information about these concepts.

\begin{lem}
A $t$-$(v, k, \lambda)$ design has a zero-sum $2$-flow if and only if
it is even.
\end{lem}

\begin{proof}
Let $(X, \mathcal{B})$ be a $t$-$(v, k, \lambda)$ design.  If $(X,
\mathcal{B})$ is even, it is sufficient to assign $+1$ to each block
in half, namely $\frac{\rho}{2}$, of the $\alpha$-parallel classes and assign
$-1$ to each block in the other half of the $\alpha$-parallel classes. 
Note that $\alpha = r/ \rho$, where $r$ is the number of appearances of
each point $x\in X$ among the blocks of the design.  
For the converse, suppose $(X, \mathcal{B})$ has a zero-sum
$2$-flow. Since for each arbitrary element $x\in X$, there exist $r$
blocks containing $x$, exactly half of these blocks have the value
$+1$ and the rest have the value $-1$. If we take all blocks with the
same value in a set, we have two sets such that in each of them
every element appears in $\frac{r}{2}$ blocks. Therefore,
$\alpha=\frac{r}{2}$ and $\rho=2$. Hence, $(X, \mathcal{B})$ is an even design.
\end{proof}

\begin{rem}\label{res1}
By \cite[Theorem 10.1]{survey2}, a resolvable $S(2, 4, v)$ exists if
and only if $v \equiv 4 ~(\mathrm{mod}~ 12)$.  Moreover, a
$2$-partitionable $\SQS(v)$ is one that can be decomposed into $S(2,4,v)$
designs. According to \cite{Hanani}, a Steiner system $S(2,4,v)$
exists if and only if $v \equiv 1$ or $4~ (\mathrm{mod} ~12)$. So, a
necessary condition for the existence of a $2$-partitionable $\SQS(v)$ is
$v \equiv 4~ (\mathrm{mod} ~12)$. For any positive
integer $n$, there exists a $2$-partitionable $\SQS(4n)$ as well as a
$2$-partitionable $\SQS(2pn+2)$, for $p\in\{7, 31, 127\}$, see 
\cite{someconstruction}.
\end{rem}

\begin{lem}
Let $(X, \mathcal{B})$ be a $2$-resolvable $\SQS(v)$. Then $(X,
\mathcal{B})$ has a zero-sum $3$-flow. Moreover, the derived triple
system $(X_x, \mathcal{B}(x))$ for any $x\in X$, also has a zero-sum
$3$-flow.
\end{lem}

\begin{proof}
We can decompose $(X, \mathcal{B})$ into $\frac{v-2}{2}~S(2, 4, v)$
designs. We know that in this case $v \equiv 4~(\mathrm{mod}~12)$, so
$\frac{v-2}{2}$ is an odd number. Using this decomposition, it is not
hard to construct a zero-sum $3$-flow for $(X,\mathcal{B})$.  For the
second part, let $x\in X$ and consider all blocks of $(X,
\mathcal{B})$ containing $x$ to construct the derived $\STS(v-1)$. Let
$y\in X\setminus\{x\}$. As we know each pair of elements of $X$
appears in any obtained $S(2, 4, v)$ exactly once; $y$ appears in all
of these $S(2, 4, v)$. By an appropriate assignment (using the values
$2, \pm1$), one can obtain a zero-sum $3$-flow on the derived
$\STS(v-1)$.
\end{proof}

\begin{rem}
By \cite{surveysqs}, the constructions of $\SQS(8)$ and $\SQS(10)$ are
unique.  We show that $\SQS(8)$ and $\SQS(10)$ admit a zero-sum
$3$-flow.  The following blocks form $\SQS(8)$, and the value from
$\{\pm 1, 2\}$ given on the right hand side of each block is the flow
assigned to that block.
\begin{equation*}
\begin{array}{cr@{\hskip2cm}cr}
1~2~4~8&\textcolor{red}{1}&3~5~6~7&\textcolor{red}{1}\\
2~3~5~8&\textcolor{red}{1}&1~4~6~7&\textcolor{red}{1}\\
3~4~6~8&\textcolor{red}{2}&1~2~5~7&\textcolor{red}{2}\\
4~5~7~8&\textcolor{red}{-1}&1~2~3~6&\textcolor{red}{-1}\\
1~5~6~8&\textcolor{red}{-1}&2~3~4~7&\textcolor{red}{-1}\\
2~6~7~8&\textcolor{red}{-1}&1~3~4~5&\textcolor{red}{-1}\\
1~3~7~8&\textcolor{red}{-1}&2~4~5~6&\textcolor{red}{-1}
\end{array}
\end{equation*}

Moreover, the blocks below form $\SQS(10)$, with the assigned flows of
a zero-sum $2$-flow specified next to the corresponding blocks.  Note
that its derived $\STS(9)$ also has a zero-sum $2$-flow.
\begin{equation*}
\begin{array}{cr@{\hskip1cm}cr@{\hskip1cm}cr}
1~2~4~5  &\textcolor{red}{1} &  1~2~3~7&\textcolor{red}{-1} &  1~3~5~8 &\textcolor{red}{1} \\
2~3~5~6&\textcolor{red}{-1}  &  2~3~4~8  &\textcolor{red}{1}  &  2~4~6~9 &\textcolor{red}{-1}\\
3~4~6~7  &\textcolor{red}{1}   &  3~4~5~9&\textcolor{red}{-1}   &  3~5~7~0  &\textcolor{red}{1}\\
4~5~7~8&\textcolor{red}{-1} &  4~5~6~0  &\textcolor{red}{1} &  1~4~6~8&\textcolor{red}{-1}\\
5~6~8~9  &\textcolor{red}{1} &  1~5~6~7&\textcolor{red}{-1} &  2~5~7~9  &\textcolor{red}{1} \\
6~7~9~0&\textcolor{red}{-1} &  2~6~7~8   &\textcolor{red}{1} &  3~6~8~0 &\textcolor{red}{-1}\\
1~7~8~0 &\textcolor{red}{1} &   3~7~8~9&\textcolor{red}{-1} &   1~4~7~9  &\textcolor{red}{1}\\
1~2~8~9&\textcolor{red}{-1} &  4~8~9~0 &\textcolor{red}{1} &  2~5~8~0&\textcolor{red}{-1} \\
2~3~9~0 &\textcolor{red}{1} &  1~5~9~0&\textcolor{red}{-1} &  1~3~6~9   &\textcolor{red}{1}\\
1~3~4~0&\textcolor{red}{-1} &   1~2~6~0  &\textcolor{red}{1} &   2~4~7~0 &\textcolor{red}{-1}
\end{array}
\end{equation*}
\end{rem}

\begin{cor}
Every $\SQS(v)$ admits a zero-sum $k$-flow for some positive integer $k$.
\end{cor}

\begin{proof}
Since every $3$-design is also a $2$-design, by Theorem \ref{in3}, the
assertion is proved.
\end{proof}


\end{document}